\documentclass[12pt]{amsart}
\usepackage{amssymb, verbatim,url,lineno}
\usepackage{enumerate,mathtools}
\usepackage{color}
\usepackage{leftidx}

\newcommand{\D}{\mathcal{D}}

\newcommand{\I}{\mathcal{I}}

\newcommand{\norm}[1]{\left\lVert#1\right\rVert}

\date{\today}

\newtheorem{theorem}{Theorem}[section]

\newtheorem{lemma}[theorem]{Lemma}
\newtheorem{proposition}[theorem]{Proposition}

\theoremstyle{definition}
\newtheorem{definition}[theorem]{Definition}

\theoremstyle{remark}

\numberwithin{equation}{section}

\usepackage[T1]{fontenc}

\begin{document}

\title[New numerical scheme of Atangana-Baleanu fractional integral]
{New numerical scheme of Atangana-Baleanu fractional integral: an application to groundwater flow within leaky aquifer}

\author[Djida]{J.D. Djida}
\address[Djida]{African Institute for Mathematical Sciences (AIMS), P.O. Box 608, Limbe Crystal Gardens, South West Region, Cameroon. }
\email[Djida]{jeandaniel.djida@aims-cameroon.org}
\author[Area]{I. Area}
\address[Area]{Departamento de Matem\'atica Aplicada II,
E.E. Aeron\'autica e do Espazo, Universidade de Vigo,
Campus As Lagoas s/n, 32004 Ourense, Spain.}
\email[Area]{area@uvigo.es}
\author[Atangana]{A. Atangana}
\address[Atangana]{Institute for Groundwater Studies, Faculty of Natural and Agricultural Sciences, University of the Free State, 9301, Bloemfontein, South Africa.}
\email[Atangana]{abdonatangana@yahoo.fr}

\thanks{The first author is indebted to the AIMS-Cameroon 2015--2017 tutor fellowship. The work of I. Area has been partially supported by the Ministerio de Ciencia e Innovaci\'on of Spain under grant MTM2016--75140--P, co-financed by the European Community fund FEDER.}


\begin{abstract}
Many physical problems can be described using the integral equations  known as Volterra equations. There exists quite a number of analytical methods that can handle these equations for linear cases. For non-linear cases, a numerical scheme is needed to obtain an approximation. Recently a new concept of fractional differentiation was introduced using the Mittag-Leffler function as kernel, and the associate fractional integral was also presented. Up to this point there is no numerical approximation of this new integral in the literature. Therefore, to accommodate researchers working in the field of numerical analysis, we propose in this paper a new numerical scheme for the new fractional integral. To test the accuracy of the new numerical scheme, we first revisit the groundwater model within a leaky aquifer by reverting the time classical derivative with the Atangana-Baleanu fractional derivative in Caputo sense. The new model is solved numerically by using this new scheme.
\end{abstract}

\maketitle
\section{Introduction}
The concept of fractional differentiation with non-singular and non-local kernel has been suggested recently and is becoming a hot topic in the field of 
fractional calculus. The concept was tested in many fields including chaotic behaviour, epidemiology, thermal science, hydrology and mechanical 
engineering \cite{Abdon2, Ivan2015, Abdon_Nieto}. The numerical approximation of this differentiation was also proposed in \cite{Abdon_Nieto}. 
In the recent decade, the integral equations were revealed to be great mathematical tools to model many real world problems in several fields of science,
Technology and Engineering. In many research papers under some conditions (see e.g. \cite{Meerschaert, Zhuang, Hao} and references therein), it was proven that there is equivalence between a given differential equation and its integral equation associate. 

\medskip
Recently \cite{Zhuang} proposed a method based on a semi-discrete finite difference approximation in time and Galerkin finite element method in space. In this work we propose a new numerical approximation of Atangana-Baleanu integral which is the summation of the average of the given function and its fractional integral in Riemann-Liouville sense. This numerical scheme will be validate by solving the partial differential equation describing the 
subsurface water flowing within a confined aquifer model with the new derivative with fractional order in time component.\medskip

This paper is organized as follows: In Section \ref{Sec:2}, for the convenience of the reader, we recall some definitions and properties of fractional 
Calculus within the scope of Atangana-Baleanu. In Section \ref{Sec:3} a numerical approach of Atangana-Baleanu derivative with fractional order is 
introduced.
In Section \ref{Sec:4}, as application of the new numerical approximation of Atangana-Baleanu fractional integral. We study analytically and numerically
the model of groundwater flow within a leaky aquifer based upon the Mittag-Leffler function. Finally, Section \ref{Sec:5} is dedicated to our perspectives
and conclusions.

\section{Somes definitions of the Atangana-Baleanu fractional derivative and integral}\label{Sec:2}
We recall the definitions of the new derivative with non singular kernel and integral introduced by Atangana and Baleanu in the senses of Caputo and 
Riemann-Liouville derivatives \cite{Abdon2,Abdon3}.\medskip

Let $(a,b) \subset {\mathbb{R}}$ and let $u$ be a function of the Hilbert space $L^{2}(a,b)$. We define  by $u'$ the derivative of $u$ as distribution on 
$(a,b)$.
\begin{definition}
The Sobolev space of order $1$ in $(a,b)$ is defined as
\[
H^{1}(a,b)=\{ u \in L^{2}(a,b) \,\vert\, \,\, u' \in L^{2}(a,b) \}.
\]
\end{definition}

\begin{definition}\label{Atangana-Baleanu_Caputo}
Let $\alpha \in (0,1)$ and a function $u \in H^{1}(a,b)$, $b > a$. The Atangana-Baleanu fractional derivative in Caputo sense of order $\alpha$ of $u$ with
a based point $a$ is defined as
\begin{equation}\label{Atangana-Baleanu_Caputo:definition}
^{\textsc{ABC}}\D_t^{\alpha}u\,(t)= \frac{B(\alpha)}{1 - \alpha}\int_{a}^{t} u'(s)E_{\alpha}\bigg[-\frac{\alpha}{1-\alpha}\big(t - s\big)^{\alpha} \bigg]ds,
\end{equation}
where $B(\alpha)$ has the same properties as in Caputo and Fabrizio case, and is defined as 
\[
B(\alpha)=1-\alpha +\frac{\alpha }{\Gamma (\alpha )},
\]
$E_{\alpha, \beta}(\lambda z^{\alpha})$ is the Mittag-Leffler function, defined in terms of a series as the following entire function
\begin{equation}\label{def:mittag}
E_{\alpha,\beta}(z)=\sum_{k=0}^{\infty} \frac{\big(\lambda z^{\alpha}\big)^{k}}{\Gamma(\alpha k + \beta)}, \quad \alpha > 0, \quad \lambda < 
\infty,~~\text{and}~~
\quad \beta>0,
\end{equation}
$\lambda = -\alpha(1-\alpha)^{-1}$.
\end{definition}
The Mittag-Leffler functions appear in the solution of linear and nonlinear fractional differential equations \cite{MR1658022}. The above definition is 
very helpful to discuss real world problems and will also have a great advantage when using the Laplace transform with initial condition. Now let us recall
the definition of the Atangana-Baleanu fractional derivative in the Riemann-Liouville sense.\medskip

\begin{definition}\label{Atangana-Baleanu_Riemann}
Let $\alpha \in (0,1)$ and a function $u \in H^{1}(a,b), b > a$. The Atangana-Baleanu fractional derivative in the Riemann-Liouville sense of order $\alpha$
of $u$ is defined as
\begin{equation}\label{Atangana-Baleanu_Riemann:definition}
^{\textsc{ABR}}\D_t^{\alpha}u\,(t)= \frac{B(\alpha)}{1 - \alpha}\frac{d}{dt}\int_{a}^{t}u(s) E_{\alpha}\bigg[-\frac{\alpha}{1-\alpha}\big(t - 
s\big)^{\alpha} \bigg]ds.
\end{equation}
Notice that, when the function $u$ is constant, we get zero.
\end{definition}
\begin{definition}\label{Atangana-Baleanu_Integral:definition}
The Atangana-Baleanu fractional integral of order $\alpha$ with base point $a$ is defined as
\begin{equation}\label{Atangana-Baleanu_Integral:equation}
^{\textsc{AB}}\I_t^{\alpha} u(t) = \frac{1-\alpha}{B(\alpha)}u(t) + \frac{\alpha}{B(\alpha) \Gamma(\alpha)} \int_{a}^{t} u(s) (t-s)^{\alpha-1} ds.
\end{equation}
\end{definition}

\section{Numerical approach of the Atangana-Baleanu integral with fractional order}\label{Sec:3}
In this section, we will start by given the discretization for Riemmann-Liouville fractional integral \cite{Hao, Khalil}. Next, following the same idea as \cite{Zhuang, Gallegos}, we introduce a numerical scheme to discretize the temporal fractional integral, and give the corresponding error analysis which will be used to give the numerical solution of the Atangana-Baleanu fractional integral and also to obtain the solution of the modified groundwater flow within a leaky aquifer.\medskip

Let $f \in \mathcal{C}^{2}(a,b)$. Then $\int_{a}^{b}f(x)dx$ could be discretized as follows as
\begin{multline}\label{descrete1}
\int_{a_{0} = a}^{a_{n}=b}f(x)dx = \sum_{j = 0}^{n} \int_{t_{j}}^{t_{j+1}}\frac{f(t_{j+1}) + f(t_{j})}{2}dy
= \sum_{j = 0}^{n} \frac{f(t_{j+1}) + f(t_{j})}{2}\int_{t_{j}}^{t_{j+1}}dy\\
= \sum_{j = 0}^{n} \frac{f(t_{j+1}) + f(t_{j})}{2}\big(t_{j+1} - t_{j}\big).
\end{multline}
On the other hand 
\begin{equation}\label{descrete2}
\int_{a}^{b}f(x)dx = \sum_{j=0}^{n}\int_{t_{j}}^{t_{j+1}} f(t_{j+1})dy = \sum_{j=0}^{n}f(t_{j+1})\big(t_{j+1}-t_{j}\big).
\end{equation}
With the Riemann-Liouville fractional integral we have the following.\medskip

We choose $t\in [0,T]$, the fractional order is denoted by $\alpha \in (0,1)$, the step $\tau = \frac{T}{n}$, $(n\in \mathbb{N})$, and the grid points are $t_{k}$,
$k\in \{0,1,2,\dots, n\}$, where $t_{k} = k\tau$. Thus for $k\in \{0,1,2,\dots, n\}$, we have
\begin{equation}\label{descrete_Riemman}
_{0}^{\textsc{RL}}\I_t^{\alpha}f(t_{k}) = \frac{1}{\Gamma(\alpha)} \sum_{i=0}^{k-1}\int_{t_{i}}^{t_{i+1}}f(y)\big(t_{k} - y \big)^{\alpha-1}dy.
\end{equation}
If we consider the following discretization for the function $f$.
\begin{equation}\label{descrete3}
f(x) = \frac{f(x_{i+1}) + f(x_{i})}{2},
\end{equation}
by replacing \eqref{descrete3} into~\eqref{descrete_Riemman} we obtain
\begin{multline}\label{descrete4}
_{0}^{}\I_t^{\alpha} = \frac{1}{\Gamma(\alpha)} \sum_{j=0}^{k-1}\int_{t_{j}}^{t_{j+1}}\frac{f(t_{j+1}) + f(t_{j})}{2}\big(t_{k} - s\big)^{\alpha-1}ds,  \\
= \frac{\tau^{\alpha}}{\Gamma(\alpha+1)}\sum_{j=0}^{k-1}\frac{f(t_{j+1}) + f(t_{j})}{2}\bigg[(k - j)^{\alpha} - (k-j+1)^{\alpha} \bigg] + R_{k, \alpha}.
\end{multline}
Using the same approach as in \cite{Wang}, the error is given as
\begin{multline}\label{error}
R_{k,\alpha} = \frac{1}{\Gamma(\alpha)}\sum_{j=0}^{k-1}\int_{t_{j}}^{t_{j+1}}\frac{f(y) - f(t_{j+1}) - f(t_{j})}{(t_{k}-y)^{1-\alpha}}  
= \frac{1}{\Gamma(\alpha)}\sum_{j=0}^{k-1}\frac{f(y) - \frac{f(t_{j+1}) - f(t_{j})(t_{j+1}-t_{j})}{(t_{j+1}-t_{j})}}{(t_{k}-y)^{1-\alpha}} \\
= \frac{1}{\Gamma(\alpha)}\sum_{j=0}^{k-1}\frac{(t_{j+1}-t_{j})}{(t_{k}-y)^{1-\alpha}}\big( f(y) - f'(\xi) \big)dy, \quad t_{j} < \xi < t_{j+1} .
\end{multline}
{}From the Taylor series at the point $\xi$ 
\[
f(y) = f(\xi) + yf'(\xi) + \cdots
\]
we can approximate
\begin{equation*}
f(y) -f'(\xi) \approx f(\xi) +  yf'(\xi) - f'(\xi) =f(\xi) - f'(\xi)(y-1).
\end{equation*}
Now taking the norm, we get
\[
\norm{f(y) - f'(\xi)} \approx  \norm{f(\xi) - f'(\xi)(y-1)} \leq M,
\]
since the function $f$ is differentiable.\medskip
Hence 
\begin{equation}
\norm{R_{k,\alpha}} = \frac{\tau}{\Gamma(\alpha+1)}M t_{k}^{\alpha}.
\end{equation}
\begin{lemma}
Let $f \in \mathcal{C}^{2}[0,T]$. The Atangana-Baleanu fractional integral is given by
\begin{align}\label{descrete5}
_{0}^{\textsc{AB}}\I_t^{\alpha}\big(f(t_{k})\big) =  \frac{1-\alpha}{B(\alpha)}f(t_{k}) +  \frac{\alpha \tau^{\alpha}}{B(\alpha)\Gamma(\alpha + 1)}
\sum_{j=0}^{k-1}b_{j}^{\alpha}\frac{f(t_{k-j}) + f(t_{k-j+1})}{2} + R_{k,\alpha}, 
\end{align}
where $\vert R_{k,\alpha} \vert \leq K t_{k}^{\alpha}\tau$,  $k = 1,2,3,\dots, n$, and $b_{j}^{\alpha} = (j+1)^{\alpha} - j^{\alpha}$, $j = 0,1,2,\dots, n$.
\end{lemma}

\medskip
Using the second approach of the discretization \eqref{descrete2}, we obtain the following approximation
\begin{equation}\label{descrete6}
_{0}^{\textsc{AB}}\I_t^{\alpha}\big(f(t_{k})\big) =  \frac{1-\alpha}{B(\alpha)}f(t_{k}) +  \frac{\alpha \tau^{\alpha}}{B(\alpha)\Gamma(\alpha + 1)}
\sum_{j=0}^{k-1}y(t_{j+1})\big[(k-j)^{\alpha} - (k-j-1)^{\alpha} \big] +  \tilde{R}_{k,\alpha},
\end{equation}
where 
\[
\tilde{R}_{k,\alpha} = \sum_{j = 0}^{k-1}\int_{t_{j}}^{t_{j+1}}\frac{f(y) + f(t_{j+1})}{(t_{k}-y)^{1-\alpha}}dy, 
\]
and 
\[
\vert \tilde{R}_{k,\alpha} \vert \leq \frac{\tau}{\Gamma(\alpha)}t_{k}^{\alpha}\max_{0 \leq t \leq t_{k}} \vert f'(t) \vert.
\]

If the function $f$ is differentiable such that its Atangana-Baleanu fractional integral exists then 
\begin{equation}
\frac{d}{dt} ~ _{0}^{\textsc{AB}}\I_t^{\alpha}\big(f(t)\big) =  \frac{1-\alpha}{B(\alpha)}f'(t) + \frac{\alpha}{B(\alpha)\Gamma(\alpha)}\frac{d}{dt}
\int_{0}^{t}f(y)(t-y)^{\alpha-1}dy.
\end{equation}
Then for $k=0,1,2,3,\dots, n$ the Atangana-Baleanu integral can be decomposed as follows
\begin{multline}
 _{0}^{\textsc{AB}}\I_t^{\alpha}\big(f(t)\big) \\ =  \frac{1-\alpha}{B(\alpha)} \big(f(t_{k} - f(t_{k-1}) \big) + \Psi(\alpha) \int_{0}^{t_{k}}
 \frac{f(s)}{(t_{k} - y)^{1-\alpha}}ds  
- \Psi(\alpha)\int_{0}^{t_{k-1}}\frac{f(y)}{(t_{k-1}-y)^{\alpha-1}}dy  \\
= \frac{1-\alpha}{B(\alpha)} \big(f(t_{k} - f(t_{k-1}) \big) + \Psi(\alpha) \int_{0}^{\tau}\frac{f(y)}{(t_{k} - y)^{1-\alpha}}dy  
+ \Psi(\alpha)\int_{0}^{t_{k-1}}\frac{f(y+\tau) - f(y)}{(t_{k-1} - y)^{1-\alpha}}dy  \\
 = \Psi(\alpha) \int_{0}^{\tau}\frac{f(y)}{(t_{k} - y)^{1-\alpha}}dy  + ~ _{0}^{\textsc{AB}}\I_t^{\alpha}\big(\beta(t_{k-1})\big),  
\end{multline}
where $\Psi(\alpha)=\frac{\alpha}{B(\alpha)\Gamma(\alpha)}$. Nevertheless
\begin{align*}
 \int_{0}^{\tau}\frac{f(y)}{(t_{k} - y)^{1-\alpha}}dy  =  \int_{0}^{\tau}\frac{f(y)}{(t_{k} - y)^{1-\alpha}}dy  +  \int_{0}^{\tau}\frac{f(y) -
 f(\tau)}{(t_{k} - y)^{1-\alpha}}ds, 
\end{align*}
where also
\begin{align*}
\int_{0}^{\tau}\frac{f(y) - f(\tau)}{(t_{k} - y)^{1-\alpha}}ds =  \int_{0}^{\tau}\frac{f'(\xi)(y-\tau)}{(t_{k} - y)^{1-\alpha}}ds, \qquad y < \xi < \tau.
\end{align*}
Let us now assume that $f$ is two times differentiable on $[0,T]$. Then, we obtain
\[
\vert \Psi(\alpha)\int_{0}^{\tau}\frac{f(y) - f(\tau)}{(t_{k} - y)^{1-\alpha}}dy \vert \leq \frac{\tau^{1+\alpha}}{\frac{\alpha}{B(\alpha)
\Gamma(\alpha)}}b^{\alpha}_{k-1} \max_{0\leq t \leq \tau} \vert f'(t) \vert.
\]
With the above relation in hand, we can conclude that 
\begin{align*}
_{0}^{\textsc{AB}}\I_t^{\alpha}\big(\beta(t_{k-1})\big) &= \frac{1-\alpha}{B(\alpha)} \beta(t_{k-1}) + \frac{\tau^{\alpha}}{B(\alpha)\Gamma(\alpha)}
\sum_{j=1}^{k-1}b_{j-1}^{\alpha}\beta(t_{k-j}) + R'_{k,\alpha}.
\end{align*}
Here we have 
\[
\vert  R'_{k,\alpha} \vert \leq \frac{\tau^{\alpha}}{B(\alpha)\Gamma(\alpha)} t_{k-1}^{\alpha} \max_{0 \leq t \leq t_{k}} \vert f''(\xi) \vert, \qquad 0
\leq \xi < t_{k}.
\]
Therefore, we have the following relationship
\begin{equation}
_{0}^{\textsc{AB}}\I_\Delta t^{\alpha}~f(t_{k-1}) = \frac{1-\alpha}{B(\alpha)} \beta(t_{k-1}) + \frac{\tau^{\alpha}}{B(\alpha)\Gamma(\alpha)} \bigg\{
b_{k-1}^{\alpha}f(\tau)
 + \sum_{j=1}^{k-1}b_{j-1}^{\alpha}\big[f(t_{k-j + 1}) - f(t_{k-j}) \big] \bigg\}+ R^{2}_{k,\alpha}, 
\end{equation}
where $\vert  R^{2}_{k,\alpha} \vert < c_{\alpha}b_{k-1}^{\alpha}\tau^{\alpha+1} + c_{\alpha}\tau^{2} t^{\alpha}_{k-1}$ \medskip

Finally the equation above can be reformulated as
\begin{equation}
_{0}^{\textsc{AB}}\I_\Delta t^{\alpha}~f(t_{k-1}) = \frac{1-\alpha}{B(\alpha)}\frac{(f(t_{k} - t_{k-1})}{\tau} + \frac{\tau^{\alpha}}{B(\alpha)\Gamma(
\alpha)}\big\{ f(t_{k})
 + \sum_{j=1}^{k-1}b_{j}^{k}- \big(b_{j-1}^{\alpha} \big)f(t_{k-j}) \big\} + R^{2}_{k,\alpha}, 
\end{equation}
where indeed $\vert  R^{2}_{k,\alpha} \vert \leq c_{\alpha}b_{k-1}^{\alpha}\tau^{\alpha+1}$.

\section{Application to groundwater flow within Leaky aquifer based upon Atangana-Baleanu fractional derivative}\label{Sec:4}
The concept of groundwater flow within the geological formation is a very complex physical problem and has attracted the attention of several scholars from different branches of sciences and technology. In particular the model portraying the movement of this subsurface water within the medium called leaky aquifer.  In this section, we focus our attention on the model based on differentiation with non-local and non-singular kernel. To do this we consider the time derivative to be the time fractional derivative based on the Mittag-Leffler function. The new model will be analysed analytically and numerically. For analytical investigation, we shall focus on the analysis of existence and uniqueness of the solution of the new model. Then we apply the new numerical scheme to derive the numerical solution of the new model.

\subsection{Analytical solution of the flow within the leaky aquifer based upon Atangana-Baleanu fractional derivative}
Let $\Omega = (a,b)$ be an open and bounded subset of $\mathbb{R}^{n}~(n \geq 1)$, with boundary $\partial \Omega$. For a given $\alpha \in (0,1)$, and a function $\varphi(r,t) \in H^{1}(\Omega) \times [0,T]$, which represents the head, we seek $\varphi$ such that the flow of water within the leaky aquifer is governed by
\begin{equation}\label{problem_Atangana-Baleanu}
\beta^{2}~{}^{ABC}\D_{t}^{\alpha}~\varphi = \partial_{rr}\varphi + \frac{1}{r} \partial_{r}\varphi - \frac{\varphi}{\varpi^{2}},
\end{equation}
where $\varpi = KDc$ and $\beta^{2} = Sc\varpi^{-2}$, $S$ denotes the coefficient of storage, $K$ denotes the conductivity.

\medskip
The problem~\eqref{problem_Atangana-Baleanu} of groundwater flow within the leaky aquifer will be approached analytically and solved numerically using the implicit scheme.

\subsection{Existence and uniqueness of the solution of the problem}

In the following, we discuss the existence and uniqueness of solutions of the direct problem. 
\begin{equation}\label{direct_Atangana_Baleanu_leaky}
\begin{cases}
~{}^{ABC}\D_{t}^{\alpha}~\varphi(t,r) = g(t,\varphi(t,r)),\\
\varphi(r_{c}, 0) = h(r) \qquad \qquad \text{on}  \qquad [0,T]\times \partial \Omega,\\
\varphi(r_{c}, t) = \varphi_{c}\qquad \qquad \text{in}~~ \{0\}~\times 
~\Omega,\\
\end{cases}
\end{equation}
for the initial datum $\varphi_{c} \in L^{\infty}(\Omega), \varphi_{c} > 0$, and where 
\begin{equation}\label{function_g}
g\big(r,t, \varphi(r,t)\big) = \beta^{-2} \bigg( \partial_{rr}\varphi(t,r) + \frac{1}{r} \partial_{r}\varphi(t,r) - \frac{\varphi(t,r)}{\varpi^{2}} \bigg).
\end{equation}

Given $\varphi_{c} \in L^{\infty}(\Omega),~ \varphi_{c} > 0$, a solution of \eqref{direct_Atangana_Baleanu_leaky} is a positive function $\varphi \in H^{1}(\Omega) \times [0,T]$ such that if we apply the fractional integral defined in \eqref{Atangana-Baleanu_Integral:equation} to \eqref{direct_Atangana_Baleanu_leaky} it yields
\begin{equation}\label{solution_1}
\varphi(t,r) = \varphi(0) + \frac{1-\alpha}{B(\alpha)}g(t,\varphi(t,r)) + \frac{\alpha}{B(\alpha) \Gamma(\alpha)} \int_{0}^{t} g(s,\varphi(s,r)) (t-s)^{\alpha-1} ds,
\end{equation}
for all $t \in [0,T]$ \medskip

We want to use the contraction mapping theorem, so for this purpose we need to build a closed set $\mathcal{\varepsilon}$ of $H^{1}(\Omega) \times [0,T]$ such that the nonlinear operator $g$ be a contraction which maps $\mathcal{\varepsilon}$ into itself. \medskip

So we first show that $g$ is a contraction mapping.

\begin{proposition}
The nonlinear operator $g \in H^{1}(\Omega) \times [0,T]$ is locally Lipschitz.
\end{proposition}
\begin{proof}
We consider two bounded functions $\varphi$ and $\vartheta$ in $\in H^{1}(\Omega) \times [0,T]$. Then,
\begin{align*}
\norm{g(\varphi) - g(\vartheta)}_{H^{1}} = \beta^{-2} \norm{ \partial_{rr}\varphi -  \partial_{rr}\vartheta + r^{-1} \big( \partial_{r}\varphi-  \partial_{r}\vartheta  \big) - \varpi^{-2}\big( \varphi- \vartheta \big)}.
\end{align*}
We have by the triangular inequality
\begin{align*}
\norm{g(\varphi) - g(\vartheta)}_{H^{1}} \leq  \beta^{-2} \bigg \{ \norm{ \partial_{rr}\varphi -  \partial_{rr}\vartheta } + \norm{ r^{-1} \big( \partial_{r}\varphi-  \partial_{r}\vartheta  \big)} - \varpi^{-2}\norm{ \varphi- \vartheta \big)} \bigg \}.
\end{align*}
As the derivative operator satisfy the Lipschitz conditions in $H^{1}$, hence there exist two positive constants $c_{1}$ and $c_{2}$ such that
\begin{multline*}
\norm{g(\varphi) - g(\vartheta)}_{H^{1}} \leq  \beta^{-2}c_{1} \norm{ \varphi - \vartheta } + \beta^{-2}c_{2} \norm{r^{-1}}\norm{\varphi- \vartheta} - \beta^{-2}\varpi^{-2}\norm{ \varphi- \vartheta },\\
\leq C \norm {\varphi- \vartheta },
\end{multline*}
where $C = |\beta^{-2} (c_{1} + c_{2}\norm{r^{-1}} - \varpi^{-2})| < 1$. Then $g$ is a contraction mapping.
\end{proof}

\begin{theorem}\label{theo1}
For a given initial datum $\varphi_{c} \in L^{\infty}(\Omega),~ \varphi_{c} > 0$, there exists an unique positive solution $\varphi$ of \eqref{direct_Atangana_Baleanu_leaky} on $ H^{1}(\Omega) \times [0,T]$, for all $t < T < \infty$, and 
\[
\lim_{t \to T} \norm{\varphi(t,r)}_{L^{\infty}} \to \infty.
\]
\end{theorem}
\begin{proof}
We shall follow the idea of \cite{Song}. Since the nonlinear operator $g$ is locally Lipschitz, for  $\bar{\varphi_{c}} = \norm{\varphi_{c}}_{H^{1}}$
there exists $C_{\bar{\varphi_{c}}}$ such that $0 < C_{\bar{\varphi_{c}}} < \infty$ and
\[
\norm{g(\varphi) - g(\vartheta)}_{H^{1}} \leq  C_{\bar{\varphi_{c}}} \norm {\varphi- \vartheta }_{H^{1}}.
\]
Let $T_{1}>0$ be a constant such that $T_{1} < \frac{1}{C_{\bar{\varphi_{c}}}}$.

\medskip
Set 
\[
\varepsilon = \big\{ \varphi \in H^{1}(\Omega) \times [0,T];~~\norm {\varphi}_{H^{1}} \leq \bar{\varphi_{c}}, \text{for all}~ t \in [0,T_{1}] \big\},
\]
endowed with the norm 
\[
\norm{\varphi}_{\varepsilon} = \sup_{0 \leq t \leq T_{1}}\norm {\varphi}_{H^{1}},
\]
then $\varepsilon$ is a closed convex subset of $H^{1}(\Omega) \times [0,T]$.

Now we consider the following associated problem of~\eqref{solution_1} defined on $\varepsilon$,
\begin{equation}\label{solution_2}
\psi(\varphi) = \varphi_{c} + \frac{1-\alpha}{B(\alpha)}g(t,\varphi(t,r)) + \frac{\alpha}{B(\alpha) \Gamma(\alpha)} \int_{0}^{t} g(s,\varphi(s,r)) (t-s)^{\alpha-1} ds.
\end{equation}
For all $\varphi \in \varepsilon$ and $t \geq 0$ we have
\begin{align*}
& \norm{\psi(\varphi)}_{\varepsilon} = \sup_{0 \leq t \leq T_{1}} \norm{ \varphi_{c} + \frac{1-\alpha}{B(\alpha)}g(t,\varphi(t,r)) + \Psi(\alpha) \int_{0}^{t} g(s,\varphi(s,r)) (t-s)^{\alpha-1} ds}_{H^{1}}\\
&\leq \norm{\varphi_{c}}_{H^{1}} + \norm{ \frac{1-\alpha}{B(\alpha)}g(t,\varphi(t,r))}_{H^{1}} + \frac{\alpha}{B(\alpha) \Gamma(\alpha)} \int_{0}^{t} \norm{ g(s,\varphi(s,r))}_{H^{1}}ds\\
&\leq \norm{\varphi_{c}}_{H^{1}} + \int_{0}^{t} \norm{ g(s,\varphi(s,r))}_{H^{1}}ds.
\end{align*}
But since $g$ is locally Lipschitz for all $s \in [0,T_{1}]$, it follows that
\begin{align*}
& \norm{\psi(\varphi)}_{\varepsilon} \leq \norm{\varphi_{c}}_{H^{1}} + \int_{0}^{t} \big(C_{\bar{\varphi_{c}}} \norm{\varphi(s,r)}_{H^{1}} + c_{3} \big)~ds.\\
\end{align*}
Thus for all $\varphi_{1}$, $\varphi_{2} \in \varepsilon $
\begin{align*}
& \norm{\psi(\varphi_{1}) - \psi(\varphi_{2})}_{\varepsilon} = \sup_{0 \leq t \leq T_{1}}T_{1}C_{\bar{\varphi_{c}}} \norm{\varphi_{1} - \varphi_{2}}_{\varepsilon}.\\
\end{align*}

This shows that $\psi$ is a contraction mapping in $\varepsilon$. Thus $\psi$ has a fixed point which is a solution to~\eqref{direct_Atangana_Baleanu_leaky}.

\medskip
Now let us show that the problem \eqref{direct_Atangana_Baleanu_leaky} has an unique solution.

\medskip
Let $\varphi_{1}$, $\varphi_{2} \in H^{1}$ be two solutions of \eqref{direct_Atangana_Baleanu_leaky} and let
$\varphi = \varphi_{1} - \varphi_{2}$. Then 
\begin{equation*}
\varphi =  \frac{1-\alpha}{B(\alpha)} \bigg( g(t,\varphi_{1}(t,r))-  g(t,\varphi_{2}(t,r)) \bigg) \\ + \frac{\alpha}{B(\alpha) \Gamma(\alpha)} \int_{0}^{t} \bigg( g(s,\varphi_{1}(s,r))-  g(s,\varphi_{2}(s,r)) \bigg)ds.
\end{equation*}
By thanking the norm on both sides,
\begin{multline*}
 \norm{\varphi} \leq \frac{1-\alpha}{B(\alpha)} \norm{ g(t,\varphi_{1}(t,r))-  g(t,\varphi_{2}(t,r)) } + \frac{\alpha}{B(\alpha) \Gamma(\alpha)} \int_{0}^{t} \norm{ g(s,\varphi_{1}(s,r))-  g(s,\varphi_{2}(s,r)) }ds\\
 \leq T_{1}C_{\bar{\varphi_{0}}} \int_{0}^{t} \norm{\varphi_{1}(s,r)}_{H^{1}}~ds.
\end{multline*}
By the Gronwall inequality~\cite{Song}, the result follows.
\end{proof}

\subsection{Exact solution of the flow within the leaky aquifer based upon Atangana-Baleanu fractional derivative}
In the model here discussed for water flow in the leaky aquifer, the head $\varphi(r,t)$, which appears in \eqref{problem_Atangana-Baleanu}, is assumed 
to be governed by the one-dimensional time-fractional differential equation involving the Atangana-Baleanu fractional derivative. \medskip

Applying Laplace transform to \eqref{problem_Atangana-Baleanu}, the fundamental solution $ \tilde{\varphi(r,p)}$, $p \in \mathbb{N}$ results to be:
\begin{equation*}
 -\beta^{2} \mathcal{L}_{t} \bigg[ ~{}^{ABC}\D_{t}^{\alpha}~\varphi \bigg](p) +  \mathcal{L}_{t} \bigg[ \partial_{rr}\varphi \bigg](p) + 
 \mathcal{L}_{t} \bigg[ \frac{1}{r} \partial_{r}\varphi\bigg](p) - \varpi^{-2} \mathcal{L}_{t} \big[\varphi \big](p) = 0,
\end{equation*}
where $\mathcal{L}_{t} := \tilde{\varphi}$ denotes the Laplace transform. Replacing each term by its value, we get
\begin{equation*}
-\frac{B(\alpha)\beta^{2}\big[p^{\alpha}\varphi - p^{\alpha-1}\varphi_{c} \big]}{(1-\alpha)p^{\alpha} + \alpha} +  
\partial_{rr} \tilde{\varphi} + \frac{1}{r}\partial_{r}\tilde{\varphi} - \varpi^{-2} \tilde{\varphi}
 =  0.  
\end{equation*}
Hence the following differential equation in the form holds
\begin{equation}\label{diff_bessel_laplace}
r^{2} \partial_{rr} \tilde{\varphi} + r~\partial_{r}\tilde{\varphi} - qr^{2}\tilde{\varphi} = 0,
\end{equation}
with $\varphi_{c} \approx 0$, where
\[
q = \frac{B(\alpha)\beta^{2}p^{\alpha}}{(1-\alpha)p^{\alpha} + \alpha} + \varpi^{-2}.
\]
Since $q$ is positive, the exact solution of the differential equation~\eqref{diff_bessel_laplace} is given in terms of Bessel function of
the first kind, $J_0$ and modified kind, $K_0$ as
\begin{equation}\label{solution_laplace1}
\tilde{\varphi}(r,p) = AJ_{0}(r\sqrt{q}) + BK_0(r\sqrt{q}),
\end{equation} 
where $A$ and $B$ are the constants and $J_0$, $K_0$ respectively given as 
\[
J_{\nu}(z) = \sum_{k=0}^{\infty}\frac{(-1)^k}{k!\Gamma(k+\nu+1)}\left( \frac{z}{2}\right)^{2k+\nu}  \text{ and } K_{\nu}(z) =
\frac{\pi}{2}(-i)^{\nu}\left(\frac{-J_{\nu}(z)+J_{-\nu}(z)}{\sin(\nu \pi)}\right).
\]

Using the boundary condition in \eqref{solution_laplace1}, we obtain $B=0$, then the solution is reduced to 
 \begin{equation}\label{solution_laplace2}
\tilde{\varphi}(r,p) = AJ_{0}(r\sqrt{q}) =  A\sum_{k=0}^{\infty}\frac{(-1)^k}{k!\Gamma(k+\nu+1)}\left( \frac{r\sqrt{q}}{2}\right)^{2k+\nu}.
 \end{equation}
Due to the difficulties to obtain the inverse Laplace transform of \eqref{solution_laplace2}, we therefore propose to obtain the approximate solution 
of \eqref{direct_Atangana_Baleanu_leaky} by using the proposed numerical approximation of the Atangana-Baleanu integral.

\subsection{Numerical analysis of the ground water flow within the leaky aquifer based upon Atangana-Baleanu fractional derivative}
To achieve this, we revert the fractional differential equation to the fractional integral equation using the link between the Atangana-Baleanu derivative
and the Atangana-Baleanu integral to obtain
\begin{equation}
\varphi(r,t) - \varphi(r_{c},0) =  ~~_{0}^{AB}\I_{t}^{\alpha}~g\big(r,t, \varphi(r,t)\big).
\end{equation}
For some positive and large integers $M = N = 350$, the grid sizes in space and time is denoting respectively by $\xi = 1/M$ and $\tau = 1/N$. The grid 
points in the space interval $(0,1]$ are the numbers $r_{i} = i\xi$, $i = 1,2,\dots,M$ and the grid points in the time interval $[0,1]$ are the numbers
$t_{k} = k \tau$, $k = 0,1,2,\dots,N$. The value of the function $\varphi$ at the grid points are denoted by $\varphi_{i}^{k} = 
\varphi(r_{i},t_{k})$.\medskip

Using the implicit finite differences method, a discrete approximation of $g\big(r, t, \varphi(r,t) \big)$ given by~\eqref{function_g} can be obtained as 
follows
\begin{equation}\label{descrite_function_g}
g \big(r_{i}, t_{k}, \varphi(r_{i},t_{k})\big) = \frac{\beta^{-2}}{\xi^2} \bigg\{ \big(1 + \frac{1}{2i}\big)\varphi_{i+1}^{k} - \big(2 + 
\frac{\xi^{2}}{\varpi^{2}} \big)  \varphi_{i}^{k} + \big(1 - \frac{1}{2i}\big)\varphi_{i-1}^{k} \bigg\}.
\end{equation}
In order to use the numerical approximation proposed in this work, discrete solution of ~\eqref{direct_Atangana_Baleanu_leaky} is then given as
follows\medskip
\begin{equation}\label{solutiona}
\varphi_{i}^{k+1} - \varphi_{c}^{0} = ~{}^{AB}\I_{t}^{\alpha}~g(r_{i},t_{k}, \varphi(r_{i},t_{k+1})).
\end{equation}
 
\begin{multline}
_{0}^{\textsc{AB}}\I_{t}^{\alpha}~g = \frac{1-\alpha}{B(\alpha)}g \big(r_{i}, t_{k+1}, \varphi(r_{i},t_{k+1})\big)  \\
 + \Psi(\alpha)
\sum_{j=0}^{k}\frac{b_{j}^{\alpha}}{2} \bigg( 
g\big(r_{i},t_{k-j}, \varphi_{i}(r_{i},t_{k-j})\big) + g\big(r_{i},t_{k-j+1}, \varphi(r_{i},t_{k-j+1}\big)\bigg).
\end{multline}
Therefore the numerical approximation can be given as follows
\begin{multline}\label{solutionb}
\varphi_{i}^{k+1} - \varphi_{c}^{0} = \frac{(1-\alpha)}{B(\alpha)}\frac{\beta^{-2}}{\xi^2} \bigg\{ \big(1 + \frac{1}{2i}\big)\varphi_{i+1}^{k+1} - 
\big(2 + \frac{\xi^{2}}{\varpi^{2}} \big)  \varphi_{i}^{k+1} + \big(1 - \frac{1}{2i}\big)\varphi_{i-1}^{k+1} \bigg\} \\
 + \Psi(\alpha)\frac{\beta^{-2}}{\xi^2}
\sum_{j=0}^{k}\frac{b_{j}^{\alpha}}{2} \bigg\{ 
 \big(1 + \frac{1}{2i}\big)\varphi_{i+1}^{k-j} - \big(2 + \frac{\xi^{2}}{\varpi^{2}} \big)  \varphi_{i}^{k-j} + \big(1 - 
 \frac{1}{2i}\big)\varphi_{i-1}^{k-j}  \\
+ \big(1 + \frac{1}{2i}\big)\varphi_{i+1}^{k-j+1} - \big(2 + \frac{\xi^{2}}{\varpi^{2}} \big)  \varphi_{i}^{k-j+1} + \big(1 - 
\frac{1}{2i}\big)\varphi_{i-1}^{k-j+1} \bigg\} .
\end{multline}

In order obtain the plots of the solution given by \eqref{solutionb}, we shall consider $350$ equidistant nodes in $[0,1]$. For the value $\alpha= 1/2$, 
the approximate solution found by using the numerical method of the Atangana-Baleanu integral proposed is obtained. Figure \ref{fig:fig_a} shows the 
approximate solution for the different time step $k = 0$, $k = 20$ and $k = 50$ respectively.
\begin{figure}[ht!]
\centering 
\begin{center}
  \includegraphics[width=0.30\textwidth]{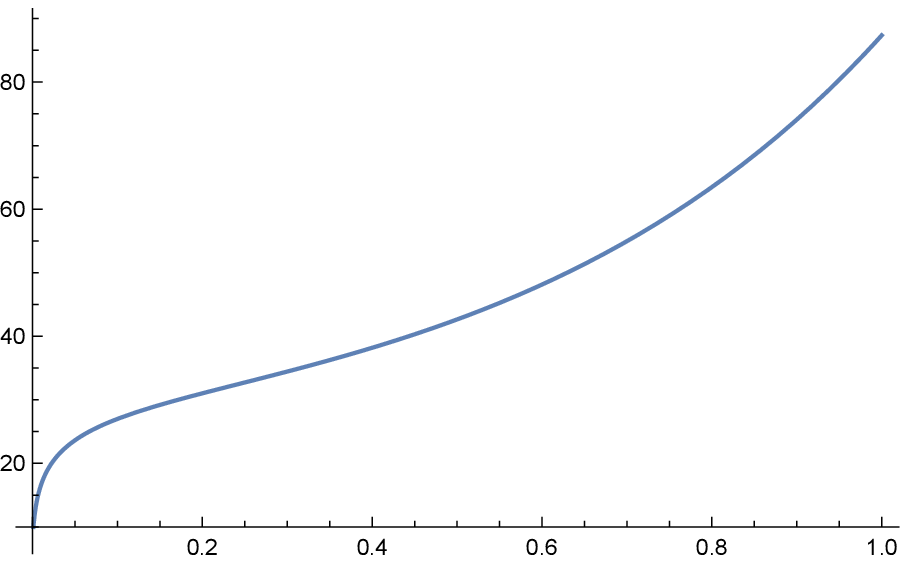} $\quad$
  \includegraphics[width=0.30\textwidth]{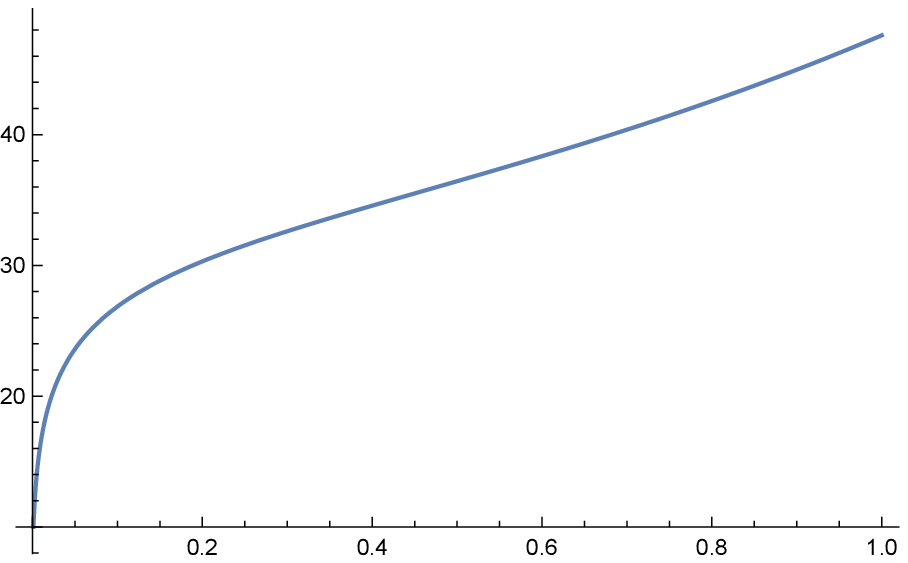} $\quad$
  \includegraphics[width=0.30\textwidth]{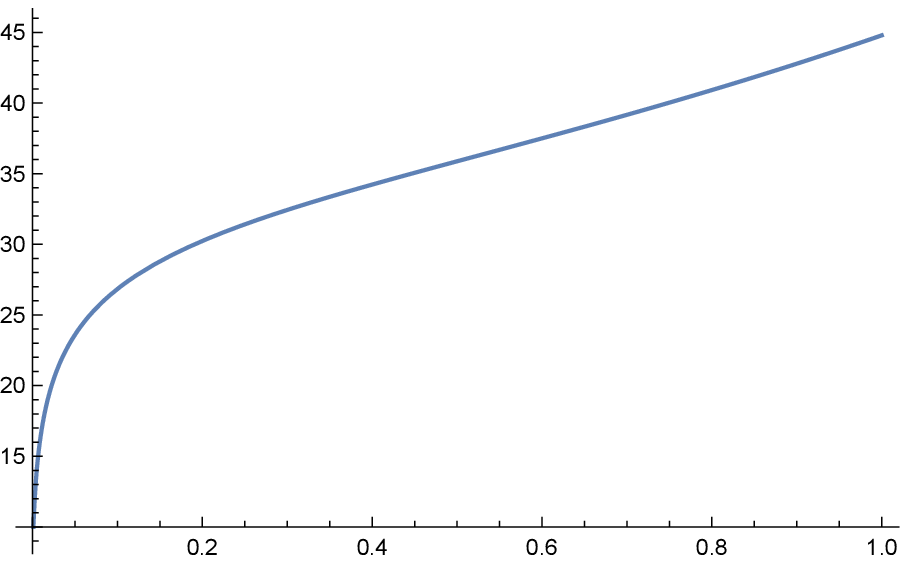}
\end{center}
\caption{Numerical solution of \eqref{solutionb} for the specific value of the parameters $\alpha=1/2$, and different time step $k = 0$, $k = 20$ and $k =
50$ respectively.}
\label{fig:fig_a}
\end{figure}

Moreover, to have an overview of the variation of flow or the behaviour of the function $\varphi$ in a finite time in terms of the parameter $\alpha$, we 
consider two different time steps $k = 0$ and $k =  50$. For this purpose, Figure \eqref{fig:fig_b} gives us the approximate result using the method
proposed. We would like to notice that for larger number of nodes in $[0,1]$ the better approximated solution is obtained in the whole interval.
\begin{figure}[ht!]
\centering 
 \includegraphics[width=0.35\textwidth]{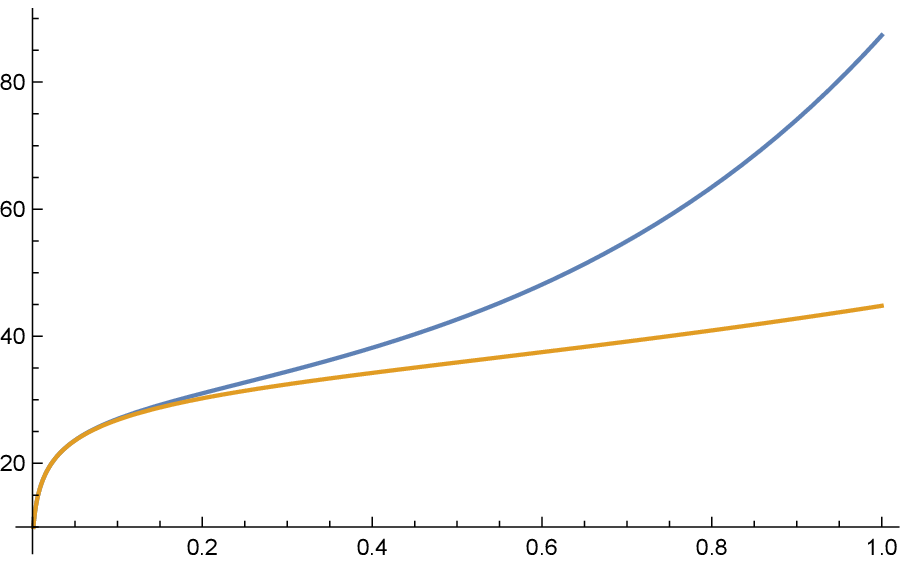} $\quad$
\includegraphics[width=0.35\textwidth]{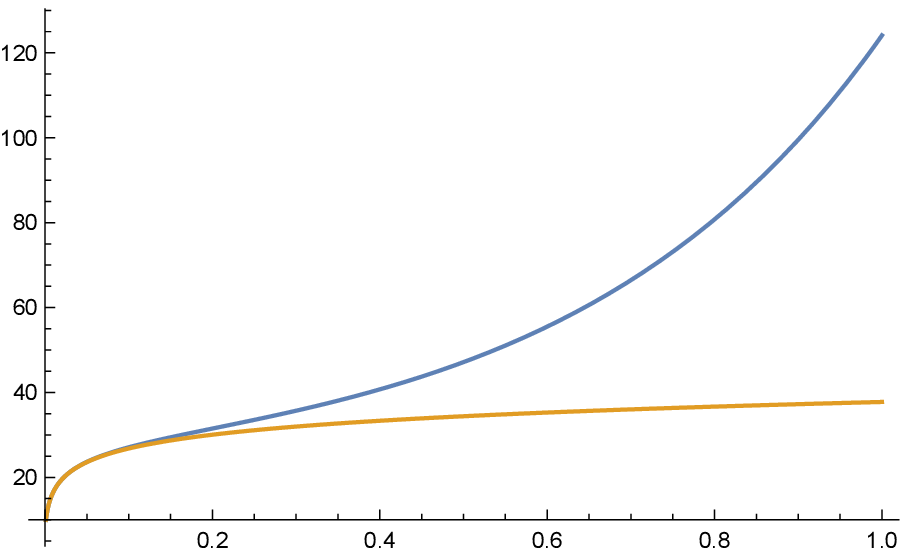}
\caption{Approximate solution of \eqref{solutionb} for initial time step $k=0$ and $\alpha = 1/2$---blue---  and $\alpha = 9/10$ ---orange--- (graph in the 
left). Approximate solution of \eqref{solutionb} for  time step $k=50$ and $\alpha = 1/2$---blue---  and $\alpha = 9/10$ ---orange--- (graph in the right)}
\label{fig:fig_b}
\end{figure}

Next in Figure \ref{fig:fig_c} we show the approximate solution of \eqref{solutionb} by using the numerical approximation of the Atangana-Baleanu method 
proposed for different time steps $k = 0$ and $k = 50$ for $\alpha = 1/2$ and $\alpha = 9/10$ in $[0,1]$.
\begin{figure}[ht!]
\centering 
 \includegraphics[width=0.35\textwidth]{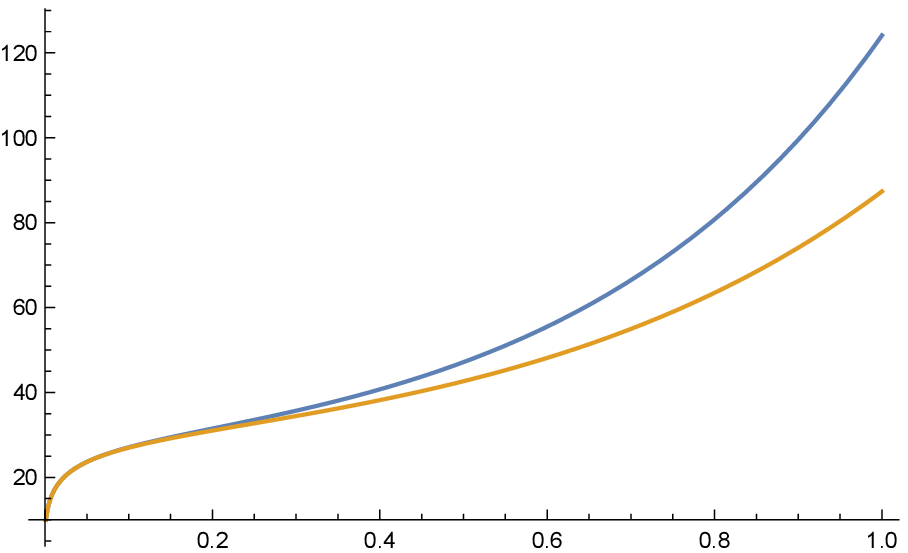} $\quad$
\includegraphics[width=0.35\textwidth]{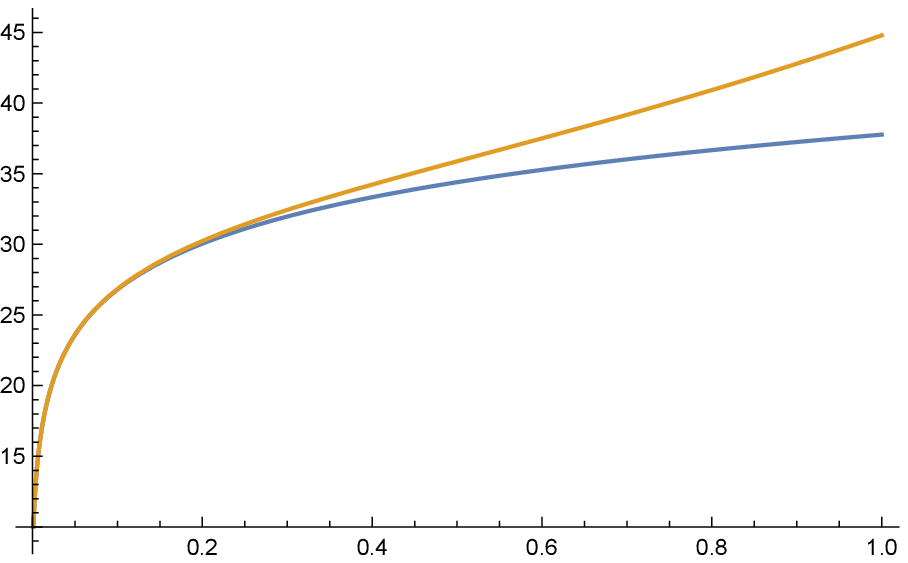}
\caption{Approximate solution of \eqref{solutionb} for $\alpha = 1/2$, $k = 0$ ---blue--- and $k = 50$---orange--- ~(graph in the left). Approximate 
solution of \eqref{solutionb} for $\alpha = 9/10$, $k = 0$ ---blue--- and $k = 50$---orange--- ~(graph in the right).}
\label{fig:fig_c}
\end{figure}

\section{Conclusions}\label{Sec:5}
The new scheme of the fractional integral in the sense Atangana-Baleanu has been proposed in this work. The error analysis of the novel scheme was 
successfully presented and the error obtained shows that the scheme is highly accurate. A new model of groundwater flowing within a leaky aquifer was 
suggested using the concept of fractional differentiation based on the generalised Mittag-Leffler function in order to fully introduce into mathematical 
formulation the complexities of the physical problem as the flow taking place in a very heterogeneous medium. The Mittag-Leffler operator provide more 
natural observed fact than the more used power law. The new model was analysed, as the uniqueness and the existence of the solution was investigated with 
care. To further access the accuracy of the proposed numerical scheme, we solved the new model numerically using this suggested scheme. Some simulations 
have been presented for different values of fractional order. We strongly believe that this numerical scheme will be applied in many fields of science, 
technology and engineering for those problems based on the new fractional calculus.


\end{document}